\documentclass[11pt, letterpaper]{amsart}
\usepackage[utf8]{inputenc}

\usepackage{amssymb, amsmath, amsthm, wasysym, mathrsfs}
\usepackage{hyperref}
\usepackage{cleveref}
\usepackage[alphabetic,lite]{amsrefs}
\usepackage{tikz-cd}
\usepackage{comment}
\usepackage{thmtools}
\usepackage{setspace}
\usepackage{mathtools}
\usepackage{enumitem}
\usepackage{thm-restate}

\usepackage{caption}
\usepackage{subcaption}
\usepackage{float}

\newtheorem{lemma}{Lemma}[section]
\newtheorem{lemma*}{Lemma}
\newtheorem{theorem}[lemma]{Theorem}

\newtheorem{prop}[lemma]{Proposition}
\newtheorem{cor}[lemma]{Corollary}

\newtheorem{claim*}{Claim}

\newtheorem{defn}[lemma]{Definition}

\theoremstyle{definition}

\newtheorem*{lem}{Acknowledgements}
\newtheorem{rmk}[lemma]{Remark}
\theoremstyle{plain}

    \newtheoremstyle{TheoremNum}
        {\topsep}{\topsep}              %%% space between body and thm
        {\itshape}                      %%% Thm body font
        {}                              %%% Indent amount (empty = no indent)
        {\bfseries}                     %%% Thm head font
        {.}                             %%% Punctuation after thm head
        { }                             %%% Space after thm head
        {\thmname{#1}\thmnote{ \bfseries #3}}%%% Thm head spec
    \theoremstyle{TheoremNum}

\newcommand{\Q}{{\mathbb Q}}
\newcommand{\R}{{\mathbb R}}
\newcommand{\Z}{{\mathbb Z}}

\numberwithin{equation}{section}
\numberwithin{table}{section}
\title{Profinite rigidity and hyperbolic four-punctured sphere bundles over the circle}
\author{Tamunonye Cheetham-West}
\date{Fall 2024}
\address{Department of Mathematics \\ Yale University \\ New Haven, CT, 06511}
  \email{tamunonye.cheetham-west@yale.edu}

\begin{document}

\maketitle
\pagestyle{plain}
\begin{abstract}
    We show that hyperbolic four-punctured $S^2-$bundles over $S^1$ are distinguished by the finite quotients of their fundamental groups among all 3-manifold groups. To do this, we upgrade a result of Liu to show that the topological type of a fiber is detected by the profinite completion of the fundamental group of a fibered hyperbolic 3-manifold. 
\end{abstract}
\bibliographystyle{alpha}

\section{Introduction}
\noindent For a compact, connected 3-manifold $M$, it is interesting to understand what properties of $M$ are detected by the profinite completion $\widehat{\pi_1(M)}$; i.e. the inverse limit of finite quotients of $\pi_1(M)$. If $M$ is the only compact 3-manifold with profinite completion $\widehat{\pi_1(M)}$, we say that $\pi_1(M)$ is {\it profinitely rigid among 3-manifold groups}. 
\medbreak For example, if two compact 3-manifolds $M$ and $N$ have $\widehat{\pi_1(M)}\cong\widehat{\pi_1(N)}$, it is a consequence of work of Lott and L{\"u}ck \cite{Luck1995} \cite{LuckApprox} on the first $L^2-$Betti number of compact 3-manifolds that $M$ is irreducible if and only if $N$ is. Furthermore, for an irreducible manifold, Wilton and Zalesskii \cite{WZ1} prove that the profinite completion determines whether a manifold is finite-volume hyperbolic. Remarkably, Liu \cite{Y} showed that any set of finite-volume hyperbolic 3-manifolds whose fundamental groups have a fixed common profinite completion is always finite. 
\medbreak Throughout, $\Sigma_{g,p}$ will refer to a genus $g$ surface with $p$ punctures. Bridson, Reid, and Wilton \cite{BRW} show that whenever the Mapping Class Group $Mod(\Sigma_{g,p})$ of a surface $\Sigma_{g,p}$ is {\it omnipotent} and has the {\it Congruence Subgroup Property} (see Section~\ref{modS} for definitions), hyperbolic $\Sigma_{g,p}$ bundles over $S^1$ with first Betti number 1 with the same finite quotients have a common finite-sheeted cyclic cover. In particular, they show that all hyperbolic $\Sigma_{1,1}$ bundles are profinitely rigid among 3-manifolds. We combine this circle of ideas with the work of Liu \cite{Y} (see Section~\ref{sec:liu} for a description) to prove:
\begin{theorem}\label{profiniterig}
Let $M$ be a hyperbolic $\Sigma_{0,4}$ bundle over $S^1$. Then $\pi_1(M)$ is profinitely rigid among 3-manifold groups. 
\end{theorem}
In particular, the work of Liu \cite{Y} allows us to remove the first Betti number 1 condition. 
\begin{theorem}\label{toptype}
 Let $M$ and $N$ be finite-volume hyperbolic manifolds with $\widehat{\pi_1(M)}\cong\widehat{\pi_1(N)}$. Liu's Thurston-norm and fiber class preserving isomorphism $H^1(N,\Z)\to H^1(M,\Z)$ induced by an isomorphism $\Phi:\widehat{\pi_1(M)}\to\widehat{\pi_1(N)}$ sends fibered classes to fibered classes where the corresponding fiber surfaces have the same topological type.
\end{theorem}
Theorem~\ref{profiniterig} gives many different examples of hyperbolic 3-manifolds that are distinguished from every other compact 3-manifold. One notable example is the complement of the three-component chain link in $S^3$ (shown in Figure~\ref{fig:magic}). This manifold is also called the {\it magic} manifold \cite{GordonWu} and is denoted as $6^3_1$ in \cite{rolfsen2003knots}. This follows immediately from Theorem~\ref{profiniterig} and Lemma 2.6(1)\cite{EikoKin} which proves the magic manifold is a $\Sigma_{0,4}-$bundle.

\begin{figure}
    \centering
    \includegraphics[scale=0.25]{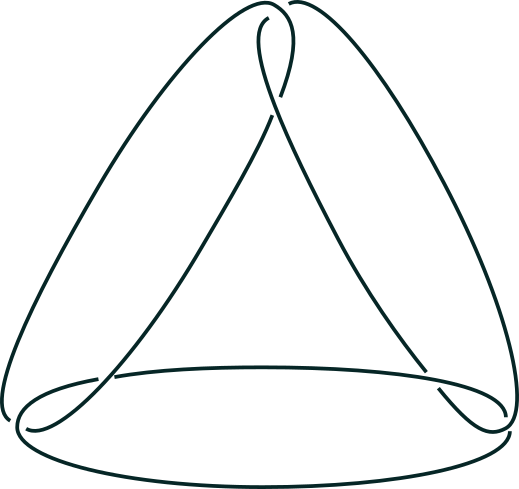}
    \caption{The magic manifold $6_1^3$}
    \label{fig:magic}
\end{figure}

\begin{cor}\label{magic}
    The magic manifold has profinitely rigid fundamental group among 3-manifold groups.
\end{cor}
Using the ideas in the proof of Theorem~\ref{profiniterig} we can also prove
\begin{theorem}\label{csp+omni}
    Let $\Sigma$ be a finite-type surface for which $Mod(\Sigma)$ is omnipotent and has Congruence Subgroup Property. If $M_\phi$ and $M_\psi$ are finite-volume hyperbolic $\Sigma$-bundles (monodromies $\phi$ and $\psi$ in $Mod(\Sigma)$) with $$\widehat{\pi_1(M_\phi)}\cong \widehat{\pi_1(M_\psi)}$$ then $M_\phi$ and $M_\psi$ have a common finite-sheeted cyclic cover and the same volume. 
\end{theorem}
Combining Theorem~\ref{csp+omni} with Theorem 8 \cite{BHMS}, one can also prove
\begin{theorem}\label{bhms}
    Let $\Sigma$ be a finite-type surface for which $Mod(\Sigma)$ has the Congruence Subgroup Property. Let $M_\phi$ and $M_\psi$ be finite-volume hyperbolic $\Sigma$-bundles (monodromies $\phi$ and $\psi$ in $Mod(\Sigma)$) with $$\widehat{\pi_1(M_\phi)}\cong \widehat{\pi_1(M_\psi)}$$ If all hyperbolic groups are residually finite, then $M_\phi$ and $M_\psi$ have a common finite-sheeted cyclic cover and the same volume. 
\end{theorem}

\begin{lem}
    {\it The author thanks Ian Agol, Autumn Kent, Chris Leininger, Rylee Lyman, Ben McReynolds, Mark Pengitore, and Zhiyi Zhang for helpful conversations about this project. The author thanks Alan Reid for helpful conversations and comments on earlier drafts of this paper. The author is also grateful for corrections to previous versions of this preprint from Martin Bridson, Biao Ma, and Ryan Spitler.}
\end{lem}
\section{Preliminaries}\label{sec:prelim}
Let $\Sigma$ be a finite-type surface. Given an orientation-preserving homeomorphism $f:\Sigma\to \Sigma$, the mapping torus of this homomorphism, $M_f$ is a 3-manifold $(\Sigma\times [0,1])/(x,0)\sim (f(x),1)$. By projecting to the second factor of the product, the 3-manifold $M_f$ is a fibration with base space $S^1$ and fiber $\Sigma$. The mapping torus admits a hyperbolic metric exactly when $f$ is {\it pseudo-Anosov} (up to isotopy) \cite{ThurstonFiber}. 
At the level of fundamental group, $\pi_1(M_f)$ is a semidirect product $\pi_1(\Sigma)\rtimes_{f_*}\Z$, where $f_*:\pi_1(\Sigma)\to \pi_1(\Sigma)$ is the automorphism induced on fundamental group by $f$. It is a theorem of Stallings \cite{StallingsFiber} that whenever $\pi_1(M)\cong A\rtimes\Z$ for $A$ a finitely generated normal subgroup, then $M$ fibers over the circle. 
\medbreak The {\it Thurston norm} \cite{Thurston1986ANF} on the homology of an orientable irreducible 3-manifold $M$ with empty or torus boundary is a seminorm $|\cdot|$ on $H_2(M,\partial M;\R)$. For $\beta\in H_2(M,\partial M;\Z)$ $$|\beta|=\min\{-\chi(S)\,|\,[S]=\beta\}$$ where $S$ is a system of surfaces $S=S_1\cup \dots\cup S_n$in $M$ representing the class $\beta$, and $-\chi(S)=\sum_{i=1}^n\max\{0,-\chi(S_i)\}$ where $\chi(S_i)$ is the Euler characteristic of $S_i$. By Poincar{\'e} duality, there is a seminorm on $H^1(M;\R)$, and any (integral) cohomology class that represents a fibration of $M$ lies in the cone over a top dimensional face of the unit norm ball called a {\it fibered face} of $M$. When $M$ is hyperbolic, the Thurston seminorm is a norm (Theorem 1 \cite{Thurston1986ANF}).
\medbreak We now consider the profinite completion $\widehat{\pi_1(M_f)}$ of the fundamental group of a fibered 3-manifold $\widehat{\pi_1(M_f)}$. One can check that the subspace topology induced on the fiber subgroup $\pi_1(\Sigma)$ coincides with the profinite topology on this subgroup (Lemma 2.2 \cite{BR}), and so there is an exact sequence
$$1\to\widehat{\pi_1(\Sigma)}\to \widehat{\pi_1(M_f)}\to\hat{\Z}\to 1$$
Moreover, Jaikin-Zapirain \cite{JZ} showed that when $M$ is a compact 3-manifold with $\widehat{\pi_1(M)}\cong\widehat{\pi_1(N)}$ for $N$ a fibered 3-manifold, then $M$ is fibered. 
\section{Liu's Theorems}\label{sec:liu}
When $M$ is a finite-volume hyperbolic 3-manifold, Liu \cite{Y} shows that there at most finitely many hyperbolic manifolds $N_1,\dots, N_k$ with $\widehat{\pi_1(N_i)}\cong\widehat{\pi_1(M)}$. To show this, Liu proves the following theorems that are crucial for this work
\begin{theorem}[\cite{Y}, Theorem 1.2, Theorem 1.3]\label{liuthm}
    Let $M, N$ be finite volume hyperbolic 3-manifolds and let $\Phi:\widehat{\pi_1(M)}\to\widehat{\pi_1(N)}$ be an isomorphism between the profinite completions of their fundamental groups. The isomorphism $\Phi$ induces an isomorphism $\Phi_*:\widehat{H}_1(M;\Z)\to \widehat{H}_1(N;\Z)$ with $\Phi_*=\hat{h}\circ \mu$ where $h:H_1(M;\Z)\to H_1(N;\Z)$ and $\mu$ denotes the scalar multiplication by a unit $\mu\in\hat{\Z}^\times$. The dual homomorphism $h^*:H^1(N;\Z)\to H^1(M;\Z)$ is Thurston-norm preserving and sends fibered classes of $N$ to fibered classes of $M$.
\end{theorem}
\begin{theorem}[\cite{Y}, Corollary 6.2, Corollary 6.3]\label{liuthm2}
    Let $M, N$ be finite volume hyperbolic 3-manifolds and let $\Phi:\widehat{\pi_1(M)}\to\widehat{\pi_1(N)}$ be an isomorphism between the profinite completions of their fundamental groups. For any connected fiber surface $\Sigma_M$ for $M$, there is a connected fiber surface $\Sigma_N$ for $N$ such that $\Phi(\widehat{\pi_1(\Sigma_M)})\cong\widehat{\pi_1(\Sigma_N)}$. Furthermore, $\chi(\Sigma_M)=\chi(\Sigma_N)$.
\end{theorem}

\section{Detecting topological type of fibers using profinite completions}
This section aims to prove Theorem~\ref{toptype}. To begin we analyze surface automorphisms and their induced actions on characteristic quotients of the surface groups. Then let $\Sigma=\Sigma_{g,n}$ be a finite-type genus $g$ surface with $n$ punctures and $f:\pi_1(\Sigma)\to\pi_1(\Sigma)$ an automorphism. Assume $f$ is induced by a pseudo-Anosov. We can replace $f$ with $f^m$ for some integer $m$ to obtain a pseudo-Anosov that fixes each puncture of $\Sigma$ (that is, fixes the conjugacy classes in $\pi_1(\Sigma)$ of peripheral loops around each puncture of $\Sigma$). We define the following numbers:
\begin{align*}
    n_f &=\text{number of fixed primitive conjugacy classes of } f \text{ in }\pi_1(\Sigma)\\
    \hat{n}_f &=\text{number of fixed primitive conjugacy classes of }\hat{f} \text{ in } \widehat{\pi_1(\Sigma)}
\end{align*}
\begin{prop}\label{keyprop}
In the setup above, for a pseudo-Anosov $f$, $n=n_f=\hat{n}_f$.
\begin{proof}
Since a pseudo-Anosov is not reducible by definition, $n=n_f$. Upon completion, each fixed conjugacy class of $f$ in $\pi_1(S)$ will give a fixed conjugacy class of $\hat{f}$ in $\widehat{\pi_1(\Sigma)}$, hence $n_f\leq \hat{n}_f$. To see that $\hat{n}_f\leq n_f$, we observe that for any element $\alpha$ of a fixed conjugacy class of $\hat{f}$ in $\widehat{\pi_1(\Sigma)}$, we can choose $t\in\widehat{\pi_1(M_{f})}$, such that $t$ and $\alpha$ commute. The element $t$ can be chosen to be a pre-image of a unit $\xi\in\hat{\Z}$ (which is a topological generator for $\hat{\Z}$) under the fixed epimorphism $\widehat{\pi_1(M_{f})}\to\hat{\Z}$). Let $H=\overline{\langle\alpha,t\rangle}$ be the closed abelian group generated by $\alpha$ and $t$. 
\medbreak We argue that this subgroup $H$ is not procyclic. First, a single topological generator $s$ for this group cannot live in the kernel $\widehat{\pi_1(\Sigma)}$ of the epimorphism $\widehat{\pi_1(M_{f})}\twoheadrightarrow \hat{\Z}$ because $t$ maps non-trivially under this epimorphism. If, on the other hand, $s$ maps non-trivially to $\hat{\Z}$ under the fixed epimorphism above, the subgroup topologically generated by $s$ also maps non-trivially to $\hat{\Z}$, and $\alpha$ cannot map non-trivially to $\hat{\Z}$ under this epimorphism. Thus, $H\leq \widehat{\pi_1(M_{f})}$ is not procyclic. Since $H$ is abelian and not procyclic, $H$ is not a closed subgroup of a free profinite group and therefore, $H$ is not projective by Lemma 7.6.3 \cite{RZ}. By Theorem 9.3 \cite{WZ1}, $H$ is conjugate into the closure of a cusp subgroup $P<\pi_1(M_{f})$. The intersection $\overline{P}\cap\widehat{\pi_1(\Sigma)}$ is generated by a peripheral element $\beta\in\pi_1(\Sigma)$. Since $\langle\alpha\rangle=H\cap \widehat{\pi_1(\Sigma)}<\overline{P}\cap\widehat{\pi_1(\Sigma)}$, $\alpha$ is in the closure of a peripheral element. Thus, $\hat{n}_f\leq n_f$. 
\end{proof}
\end{prop}
\begin{cor}\label{profinitepA}
For $S$ a finite-type surface, a homomorphism $\phi_*:\pi_1(S)\to\pi_1(S)$ is induced by a pseudo-Anosov $\phi:S\to S$ if and only if for all $n\in\Z$ the completed homomorphisms $\widehat{\phi_*^n}:\widehat{\pi_1(S)}\to \widehat{\pi_1(S)}$ do not fix a non-peripheral conjugacy class. 
\begin{proof}
The forward direction is Proposition~\ref{keyprop}. For the reverse direction observe that $\phi$ is not periodic, and $\phi$ is not reducible.
\end{proof}
\end{cor}
\begin{proof}[Proof of Theorem~\ref{toptype}]
When $M$ (and $N$) are closed, Theorem~\ref{liuthm} shows that Theorem~\ref{toptype} holds. Thus, we assume that $M$ is cusped and fix a fibration of $M$ over $S^1$ with fiber a punctured surface $\Sigma_M$ and monodromy $\varphi$. By Theorem~\ref{liuthm2}, for a fixed isomorphism $\Phi:\widehat{\pi_1(M)}\to\widehat{\pi_1(N)}$, there is a punctured surface $\Sigma_N$ (with the same complexity as $\Sigma_M$) which is the fiber of a fibration of $N$ over $S^1$ (with monodromy $\psi$) and $\Phi(\widehat{\pi_1(\Sigma_M))}\cong\widehat{\pi_1(\Sigma_N)}$. The goal is to show that $\Sigma_M$ is homeomorphic to $\Sigma_N$. The surface $\Sigma_M$ is a genus $g_M$ surface with $n_M$ punctures and the surface $\Sigma_N$ is a genus $g_N$ surface with $n_N$ punctures. Since $\chi(\Sigma_M)=\chi(\Sigma_N)$ by Theorem~\ref{liuthm2}, it is sufficient to prove that $n_N=n_M$. To establish this we argue as follows: \medbreak By Theorem~\ref{liuthm} and Theorem~\ref{liuthm2}, we get the following diagram
 \[
   \begin{tikzcd}
  1 \arrow[r] &\pi_1(\Sigma_M)\arrow[r]\arrow[d] &\pi_1(M)\arrow[r]\arrow[d] &\Z\arrow[r]\arrow[d] & 1
    \\
    1 \arrow[r] &\widehat{\pi_1(\Sigma_M)}\arrow[r]\arrow[d, "\Phi|_{\widehat{\pi_1(\Sigma_M)}}"] &\widehat{\pi_1(M)}\arrow[r]\arrow[d, "\Phi"] &\widehat{\Z}\arrow[r]\arrow[d, "\times \mu"] & 1
    \\
    1 \arrow[r] &\widehat{\pi_1(\Sigma_N)}\arrow[r] &\widehat{\pi_1(N)}\arrow[r] &\widehat{\Z}\arrow[r] & 1
    \\
    1 \arrow[r] &\pi_1(\Sigma_N)\arrow[r]\arrow[u] &\pi_1(N)\arrow[r]\arrow[u] &\Z\arrow[r]\arrow[u] & 1
    \\
 \end{tikzcd}
\]
where the $\times\mu:\hat{\Z}\to\hat{\Z}$ map is multiplication by a profinite unit $\mu$ in $\hat{\Z}$. Replace $\varphi$ and $\psi$ (as needed) with an appropriate power (say $k!$ for a large enough natural number $k$) such that both $\varphi^{k!}$ and $\psi^{k!}$ are pure mapping classes. 
\medbreak We first observe that $\hat{\varphi}^{k!}$ and $\hat{\psi}^{k!}$ have an equal number of fixed conjugacy classes in $\widehat{\pi_1(\Sigma_M)}\cong\widehat{\pi_1(\Sigma_N)}$ i.e. $\hat{n}_{\varphi^{k!}}=\hat{n}_{\psi^{k!}}$. To see this, set $t\in\widehat{\pi_1(M)}$ to be in the preimage of a unit in $\hat{\Z}$, the conjugation action of $t$ on $\widehat{\pi_1(\Sigma_M)}$ is $\hat{\varphi}$. The conjugation action of $\Phi(t)$ on $\widehat{\pi_1(\Sigma_N)}$ is $\hat{\psi}$ by the diagram above. The conjugation action on $\widehat{\pi_1(\Sigma_M)}$ by $t^{k!}$ is $\hat{\varphi}^{k!}$ and the conjugation action on $\widehat{\pi_1(\Sigma_N)}$ by $\Phi(t^{k!})$ is $\hat{\psi}^{k!}$. For every conjugacy class $\alpha$ of elements in $\widehat{\pi_1(\Sigma_M)}$ fixed by $t^{k!}$-conjugation, there is a conjugacy class $\Phi(\alpha)$ of elements in $\widehat{\pi_1(\Sigma_N)}$ fixed by $\Phi(t^{k!})$-conjugation, and vice versa. Thus, $\hat{n}_{\varphi^{k!}}=\hat{n}_{\psi^{k!}}$. By Proposition~\ref{keyprop}, $n_M=n_{\varphi^{k!}}=\hat{n}_{\varphi^{k!}}$ and $n_N=n_{\psi^{k!}}=\hat{n}_{\psi^{k!}}$. Thus, $n_M=n_N$, $g_M=g_N$, and $\Sigma_M$ is homeomorphic to $\Sigma_N$ as claimed. 
\end{proof}
\begin{rmk}[M. Bridson]\label{bridsonrmk}
    The hypothesis of hyperbolicity is necessary for the proof of Theorem~\ref{toptype}. Without hyperbolicity, for instance, every mapping class $\phi\in Mod(\Sigma_{0,3})$ gives a group $F_2\rtimes_\phi\Z$ which is also the fundamental group of a $\Sigma_{1,1}-$bundle over $S^1$. 
\end{rmk}

\section{Distinguishing hyperbolic $\Sigma_{0,4}$ bundles over $S^1$}\label{modS}
\noindent Recall that the Mapping Class Group of a finite type surface $\Sigma$, here denoted as $Mod(\Sigma)$, is the group of orientation-preserving homeomorphisms of $\Sigma$ up to isotopy. 
\begin{defn}
    The group $Mod(\Sigma)$ is omnipotent if for every independent family of (infinite-order) elements (i.e. having pairwise non-conjugate powers) $\gamma_1,\dots,\gamma_n\in Mod(\Sigma)$, there is a positive integer $\kappa$ such that for any $n-$tuple of positive integers $(e_1,\dots,e_n)$ there is a finite quotient $f:Mod(\Sigma)\twoheadrightarrow Q$ such that the order of $f(\gamma_i)$ is $\kappa e_i$ for all $1\leq i\leq n$. 
\end{defn} 

\begin{defn}
    A principal congruence quotient of $Mod(\Sigma)$ is the image of $Mod(\Sigma)$ under the canonical map $$Mod(\Sigma)\to Out(\pi_1(\Sigma))\to Out(\pi_1(\Sigma)/K)$$ where $K<\pi_1(\Sigma)$ is a characteristic subgroup of finite index. A congruence quotient $$Mod(\Sigma)\to Q$$ is a finite quotient that factors through a principal congruence quotient of $Mod(\Sigma)$.
\end{defn}
\begin{defn}
    The group $Mod(\Sigma)$ has the Congruence Subgroup Property when every finite quotient of $Mod(\Sigma)$ is a congruence quotient.
\end{defn}

\begin{theorem}\label{Mod04}
The group $Mod(\Sigma_{0,4})$ is omnipotent and has the Congruence Subgroup Property.
\begin{proof}
The group $Mod(\Sigma_{0,4})\cong PSL(2,\Z)\ltimes (\Z/2\Z)^2$ (Proposition 2.7 \cite{farb2011primer}), and this is a virtually free group. Virtually free groups are omnipotent by a theorem of Bridson-Wilton (Theorem 4.3 \cite{BW}). That $Mod(\Sigma_{0,4})$ has congruence subgroup property is a theorem of Diaz-Donagi-Harbater \cite{CSPDDH}. 
\end{proof}
\end{theorem}

The following lemma makes use of the proofs of Theorem 2.4 and Lemma 2.5 \cite{BRW}. 
\begin{lemma}\label{LiuplusBRW}
For hyperbolic 3-manifolds $M$ and $N$ that fiber over $S^1$ with $\Sigma_{0,4}$ fiber and monodromies $\varphi$ and $\psi$ respectively, let $\Phi:\widehat{\pi_1(M)}\to\widehat{\pi_1(N)}$ be an isomorphism that identifies the closures of the fiber subgroups of $M$ and $N$. The monodromies $\varphi$ and $\psi$ are conjugate elements of $Mod(\Sigma_{0,4})$, and therefore $M$ and $N$ are homeomorphic. 
\end{lemma}

We complete the proof of Theorem~\ref{profiniterig} assuming Lemma~\ref{LiuplusBRW}.
\begin{proof}[Proof of Theorem~\ref{profiniterig}]
Let $\Sigma_{0,4}\hookrightarrow M\to S^1$ be a fixed fibration. By \cite{WZ1}, \cite{JZ}, and Theorem~\ref{liuthm2}, a 3-manifold $N$ with $\widehat{\pi_1(N)}\cong\widehat{\pi_1(M)}$ will be hyperbolic and will fiber over $S^1$ with fiber a surface $S$ with Thurston norm 2. Fix an identification $\Phi:\widehat{\pi_1(M)}\to\widehat{\pi_1(N)}$. By Theorem~\ref{liuthm2}, there is a connected fiber surface $\Sigma\hookrightarrow N$ with $\Phi(\widehat{\pi_1(\Sigma_{0,4}))}\cong\widehat{\pi_1(\Sigma)}$. By Theorem~\ref{toptype}, $\Sigma$ is also homeomorphic to $\Sigma_{0,4}$, and therefore by Lemma~\ref{LiuplusBRW} $M\cong N$. 
\end{proof}
We now prove Lemma~\ref{LiuplusBRW}.
\begin{proof}[Proof of Lemma~\ref{LiuplusBRW}]
   We denote the monodromies of $M$ and $N$ by $\varphi$ and $\psi$ respectively, and we will refer to the outer automorphisms of $F_3$ (the free group on three generators) induced by $\varphi,\psi$ by the same names. Let $\Sigma,\Sigma'$ be the corresponding aligned fiber surfaces (both homeomorphic to $\Sigma_{0,4}$) for $M$ and $N$ respectively. We have the following exact sequences 
  \[
   \begin{tikzcd}
  1 \arrow[r] &\pi_1(\Sigma)\arrow[r]\arrow[d] &\pi_1(M)\arrow[r]\arrow[d] &\Z\arrow[r]\arrow[d] & 1
    \\
    1 \arrow[r] &\widehat{\pi_1(\Sigma)}\arrow[r]\arrow[d, "\Phi|_{\widehat{\pi_1(\Sigma)}}"] &\widehat{\pi_1(M)}\arrow[r]\arrow[d, "\Phi"] &\widehat{\Z}\arrow[r]\arrow[d, "\cong"] & 1
    \\
    1 \arrow[r] &\widehat{\pi_1(\Sigma')}\arrow[r] &\widehat{\pi_1(N)}\arrow[r] &\widehat{\Z}\arrow[r] & 1
    \\
    1 \arrow[r] &\pi_1(\Sigma')\arrow[r]\arrow[u] &\pi_1(N)\arrow[r]\arrow[u] &\Z\arrow[r]\arrow[u] & 1
    \\
 \end{tikzcd}
\]
where all the unlabelled vertical arrows are canonical inclusions of groups into their profinite completions. 
\medbreak The isomorphism $\Phi$ embeds $\pi_1(M)$ and $\pi_1(\Sigma)$ in $\widehat{\pi_1(N)}$ as dense subgroups of $\widehat{\pi_1(N)}$ and $\widehat{\pi_1(\Sigma')}$ respectively. We consider $\Phi(\pi_1(\Sigma))<\Phi(\pi_1(M))$. Let $K_i$ be the intersection of all subgroups of $\Phi(\pi_1(\Sigma))$ of index $\leq i$. The tower $\{K_i\,|\,i\in\mathbb{N}\}$ is a cofinal tower of characteristic subgroups of $\Phi(\pi_1(\Sigma))$. Set $L_i=\widehat{K_i}\cap\pi_1(\Sigma')$ to obtain a corresponding characteristic tower of $\pi_1(\Sigma')$. Conjugation by elements of $\Phi(\pi_1(M))$ on $\Phi(\pi_1(\Sigma))$ induces outer automorphisms in $Out(\Phi(\pi_1(\Sigma))/K_i)\cong Out(\widehat{\pi_1(\Sigma')}/\widehat{L_i})$ for all $i$. Similarly, conjugation by elements of $\pi_1(N))$ on $\pi_1(\Sigma')$ induces outer automorphisms in $Out(\pi_1(\Sigma')/L_i)\cong Out(\widehat{\pi_1(\Sigma')}/\widehat{L_i})$ for all $i$. 
\medbreak The monodromy $\varphi$ is induced for the extension $$1\to \Phi(\pi_1(\Sigma))\to \Phi(\pi_1(M))\to \Z\to 1$$ by the conjugation action of an element $t\in \Phi(\pi_1(M))$  where the rightmost $\Z$ is a dense subgroup of $\hat{\Z}$ (and $t$ maps to a generator of this $\hat{\Z}$ under the fixed profinite epimorphism $\widehat{\pi_1(N)}\twoheadrightarrow\hat{\Z}$). Likewise, the monodromy $\psi$ is induced by the conjugation action of an element $s\in\pi_1(N)<\widehat{\pi_1(N)}$ that maps to a generator of $\Z$. Since the images of $t$ and $s$ under the fixed $\widehat{\pi_1(N)}\twoheadrightarrow\hat{\Z}$ both topologically generate $\hat{\Z}$, the actions of $t$ and $s$ (and therefore the actions of $\varphi$ and $\psi$) induce outer automorphisms $\varphi_i\in Out(\Phi(\pi_1(\Sigma))/K_i)\cong Out(\widehat{\pi_1(\Sigma')}/\widehat{L_i})$ and $\psi_i\in Out(\pi_1(\Sigma')/L_i)\cong Out(\widehat{\pi_1(\Sigma')}/\widehat{L_i})$ such that $\varphi_i$ and $\psi_i$ generate conjugate cyclic subgroups $C_i< Out(\widehat{\pi_1(\Sigma')}/\widehat{L_i})$.
\medbreak Assume, for sake of contradiction, that $\varphi$ and $\psi$ are independent in $Mod(\Sigma_{0,4})$, i.e. that $\varphi$ and $\psi$ do not have conjugate powers. Since $Mod(\Sigma_{0,4})$ is omnipotent there is a finite quotient $q:Mod(\Sigma_{0,4})\twoheadrightarrow Q$ for which $o(q(\varphi))\ne o(q(\psi))$. Because $Mod(\Sigma_{0,4})$ has the Congruence Subgroup Property, $q$ factors through a principal congruence homomorphism $$q_i:Mod(\Sigma_{0,4})\to Out(\pi_1(\Sigma)/K_i)$$ for some $i\in\mathbb{N}$. By construction, $\varphi_i=q_i(\varphi)$ and $\psi_i=q_i(\psi)$. By the previous paragraph, $\varphi_i$ and $\psi_i$ generate conjugate cyclic subgroups of $Out(\Phi(\pi_1(\Sigma)/K_i)$, and this contradicts the claim that $o(q(\varphi))\ne o(q(\psi))$. Thus, $\varphi$ and $\psi$ are not independent. In particular, there are positive integers $m,n$ such that $\varphi^m$ and $\psi^{\pm n}$ are conjugate. 
\medbreak Following \cite{BRW}, apply omnipotence and CSP to the cyclic subgroup of $Mod(\Sigma_{0,4})$ generated by $\varphi$ to show that there is a congruence quotient $q:Mod(\Sigma_{0,4})\to Out(\pi_1(\Sigma)/{K_i})$ for which $mn$ divides $o(q(\varphi))=o(q(\psi))$. Since $$\frac{o(q(\varphi))}{m}=o(q(\varphi^m))=o(q(\psi^n))=\frac{o(q(\psi))}{n}$$
it follows that $m=n$. Since $\varphi,\psi$ are pseudo-Anosovs in $Mod(\Sigma_{0,4})$ with a common conjugate power, there is a mapping class $\xi$ in $Mod(\Sigma_{0,4})$ with $(\xi\varphi\xi^{-1})^m=\xi\varphi^m\xi^{-1}=\psi^m$. Since $Mod(\Sigma_{0,4})$ is a (non-free) virtually free group, it splits as a finite graph of groups with finite vertex and edge groups (Theorem 1 \cite{karrass_pietrowski_solitar_1973}) and so $Mod(\Sigma_{0,4})$ acts on a simplicial tree $T$. The elements $\xi\varphi\xi^{-1}$ and $\psi$ fix a common bi-infinite geodesic $\alpha$ in $T$. The stabilizer of $\alpha$ is cyclic (since all torsion in $Mod(\Sigma_{0,4})$ fixes a vertex in $T$) and so $\xi\varphi\xi^{-1}=\psi^{\pm}$. 

\end{proof}

\begin{comment}
\begin{theorem}\label{pureprofiniterig}
Let $M$ be a hyperbolic 4-punctured $S^2$-bundle over $S^1$ with monodromy in $PMod(\Sigma_{0,4})$, then $\pi_1(M)$ is profinitely rigid among 3-manifold groups. 
\begin{proof}
For $N$ a 3-manifold with $\widehat{\pi_1(N)}\cong\widehat{\pi_1(M)}$, we can apply Theorem 1.3 \cite{Y} and Theorem 9.3 \cite{WZ1} to see that $N$ is a hyperbolic 3-manifold that fibers over $S^1$ with fiber a surface $S$ of complexity 1. In particular, $S$ is either the four-times punctured sphere or the twice-punctured torus. Since the pA monodromy of $M$ fixes every puncture of $\Sigma_{0,4}$, it follows that $M$ is a 4-cusped hyperbolic 3-manifold. By Theorem 9.3 \cite{WZ1} and Theorem 1.1 \cite{Chagas2016Hyperbolic3G}, $N$ is 4-cusped as well, and so $S$, a fiber surface for a fibration of $N$ over $S^1$, cannot be the twice-punctured torus. Thus, $S=\Sigma_{0,4}$, and we can use Lemma~\ref{LiuplusBRW} to conclude that $N$ is homeomorphic to $M$.
\end{proof}
\end{theorem}
\end{comment}

\section{Theorem~\ref{csp+omni} and further observations}
This program for proving Theorem~\ref{profiniterig} works to show that when $Mod(\Sigma)$ is omnipotent and has the Congruence Subgroup Property, hyperbolic $\Sigma$-bundles over $S^1$ whose fundamental groups have isomorphic profinite completions are cyclically commensurable. For the proof of Lemma~\ref{LiuplusBRW}, it was crucial that two pseudo-Anosovs in $Mod(\Sigma_{0,4})$ have a common power (up to conjugacy) if and only if the two pseudo-Anosovs are conjugate. In general, however, this is not true (see \cite{rootsModS} for example). 
\begin{proof}[Proof of Theorem~\ref{csp+omni}]
Let $\Sigma$ be a finite-type surface with negative Euler characteristic for which $Mod(\Sigma)$ is omnipotent and has the Congruence Subgroup Property. By Theorem~\ref{toptype} and \cite{WZ1}, any 3-manifold $N$ with the same profinite completion as a hyperbolic $\Sigma$-bundle $M$ over $S^1$ is a hyperbolic $\Sigma$-bundle over $S^1$. An analog of Lemma~\ref{LiuplusBRW} shows that the corresponding monodromies have a common conjugate power. This common conjugate power is the monodromy of a common finite cyclic cover of  $M$ and $N$. Thus $M$ and $N$ are cyclically commensurable.
\end{proof}

In their work \cite{BHMS} on hierarchically hyperbolic quotients of mapping class groups, Behrstock-Hagen-Sisto-Martin prove
\begin{theorem}[Theorem 8, \cite{BHMS}]\label{twoomnipotence}
    Let $S$ be a connected orientable surface of finite type of complexity at least 2. If all hyperbolic groups are residually finite, then the following holds. Let $g,h\in Mod(S)$ be pseudo-Anosovs with no common proper power, and let $q\in\Q$. Then there exists a finite group $G$ and a homomorphism $\phi:Mod(S)\to G$ such that $ord(\phi(g))/ord(\phi(h))=q$. 
\end{theorem}
This theorem allows us to prove Theorem~\ref{bhms}, reducing the question of commensurability of profinitely equivalent $\Sigma-$bundles over $S^1$ when $Mod(\Sigma)$ has CSP to the residual finiteness of all hyperbolic groups.   

\begin{proof}[Proof of Theorem~\ref{bhms}]
Let $\Sigma$ be a finite-type surface with negative Euler characteristic for which has the Congruence Subgroup Property. By Theorem~\ref{toptype} and \cite{WZ1}, any 3-manifold $N$ with the same profinite completion as a hyperbolic $\Sigma$-bundle $M$ over $S^1$ is a hyperbolic $\Sigma$-bundle over $S^1$. If all hyperbolic groups are residually finite, then Theorem~\ref{twoomnipotence} above can be combined with Lemma~\ref{LiuplusBRW} to show that the corresponding monodromies will have a common conjugate power. This common conjugate power is the monodromy of a common degree finite cyclic cover of  $M$ and $N$. Thus $M$ and $N$ are cyclically commensurable manifolds with the same volume.
\end{proof}

%%%%%%%%%%%%%%%%%%%%%%%%%%%%%%%%%%%%%%%%%%%%%%%%%%%%
%%%%%%%%%%%%%%%%%%%%%%%%%%%%%%%%%%%%%%%%%%%%%%%%%%%%
%%%%%%%%%%%%%%%%%%                                           %%%%%%%%%%%%%%%%%
%%%%%%%%%%%%%%%%%%		Bibliography		%%%%%%%%%%%%%%%%%
%%%%%%%%%%%%%%%%%%                                           %%%%%%%%%%%%%%%%%
%%%%%%%%%%%%%%%%%%%%%%%%%%%%%%%%%%%%%%%%%%%%%%%%%%%%
%%%%%%%%%%%%%%%%%%%%%%%%%%%%%%%%%%%%%%%%%%%%%%%%%%%%

\bibliography{main}

\end{document}